\documentclass[12pt]{amsart}
\usepackage[all]{xy}
\usepackage{a4wide}
\usepackage{amssymb}
\usepackage{amsthm}
\usepackage{hyperref}
\hypersetup{colorlinks=true,linkcolor=blue,citecolor=magenta}
\usepackage{amsmath}
\usepackage{amscd,enumitem}
\usepackage{verbatim}
\usepackage{float}
\usepackage{color}
\usepackage[mathscr]{eucal}
\usepackage[all]{xy}
\usepackage{hyperref}
\usepackage{bbm}
\usepackage{dcolumn}
\def\={\;=\;}  \def\+{\,+\,} \def\m{\,-\,}

\theoremstyle{plain}
\newtheorem{theorem}{Theorem}

\newtheorem*{corollary*}{Corollary}
\newtheorem*{Example*}{Example}

\theoremstyle{definition}

\newtheorem*{def*}{Definition}
\newtheorem*{theorem*}{Theorem}
\newtheorem*{example}{Example}

\newtheorem*{definition*}{Definition}

\theoremstyle{remark}
\newtheorem*{remark}{Remark}

\newcommand{\TBD}{8}

\newcommand{\R}{\mathbb{R}}
\newcommand{\Q}{\mathbb{Q}}
\newcommand{\Z}{\mathbb{Z}}

\newcommand{\La}{\Lambda}
\newcommand{\la}{\lambda}
\newcommand{\tz}{\theta_0}
\newcommand{\e}{\varepsilon}

\newcommand{\binomial}[2]{\left ( \begin{matrix} #1\\ #2 \end{matrix}\right) }

\newcommand{\tn}{\widehat{n}}

\newcommand{\SL}{{\text {\rm SL}}}

\address{Department of Mathematics, 320 TMCB,
Brigham Young University, Provo, UT 84602} \email{mjgriffin@math.byu.edu}

\address{Department of Mathematics, Emory University,
Atlanta, GA 30022} \email{ken.ono@emory.edu}

\address{Department of Mathematics,  Vanderbilt University,
Nashville, TN 37240} \email{larry.rolen@vanderbilt.edu}

\address{Max Planck Institute for Mathematics,
Vivatsgasse 7 
53111 Bonn, Germany } \email{dbz@mpim-bonn.mpg.de}

\begin{document}

\title[Jensen polynomials for the Riemann zeta function 
and other sequences]
{Jensen polynomials for the Riemann zeta function and other sequences}
\author{Michael Griffin, Ken Ono, Larry Rolen, and Don Zagier}

\thanks{The first two authors acknowledge the support of the NSF (DMS-1502390 and DMS-1601306),
and the second author is grateful for the support of the Asa Griggs Candler Fund.}

\begin{abstract}  In 1927 P\'olya  proved that
the Riemann Hypothesis is equivalent to the hyperbolicity of Jensen polynomials
for the Riemann zeta function $\zeta(s)$ at its point of symmetry. This hyperbolicity 
has been proved for degrees~$d\leq 3$.  We obtain an asymptotic formula for the central 
derivatives $\zeta^{(2n)}(1/2)$ that is accurate to all orders, which allows us to prove the
hyperbolicity of a density $1$ subset of the Jensen polynomials of each degree.  {Moreover,
we establish hyperbolicity for all $d\leq 8$. These results follow from
a general theorem which models  such polynomials by Hermite polynomials.}  In the case 
of the Riemann zeta function, this proves the GUE random matrix model prediction  in  derivative aspect.
The general theorem also allows us to prove a conjecture of Chen, Jia, and Wang on the partition function.
\end{abstract}

\maketitle

\section{Introduction and Statement of Results}
Expanding on notes of Jensen, P\'olya \cite{Polya} proved that the Riemann Hypothesis~(RH) is equivalent 
to the hyperbolicity of the Jensen polynomials for the Riemann zeta function~$\zeta(s)$ at its point of 
symmetry. More precisely, he showed that the RH is equivalent to the hyperbolicity of all Jensen polynomials 
associated to the sequence of Taylor coefficients $\{\gamma(n)\}$ defined by
\begin{equation}\label{TaylorXi}
\bigl(-1+4z^2\bigr)\,\Lambda\Bigl(\frac{1}{2}+z\Bigr)\=\sum_{n=0}^{\infty} \frac{\gamma(n)}{n!}\cdot z^{2n},
\end{equation}
where $\Lambda(s)=\pi^{-s/2}\Gamma(s/2)\zeta(s)=\Lambda(1-s)$, and where we say
that a polynomial with real coefficients is {\it hyperbolic} if all of its zeros are real, and where 
the {\it Jensen polynomial of degree $d$ and shift $n$} of an arbitrary sequence 
$\{{\alpha(0)}, \alpha(1),\alpha(2),\dots\}$ of real numbers is the polynomial 
 \begin{equation}\label{JensenPolynomial}
J_\alpha^{d,n}(X)\,:=\,\sum_{j=0}^d \binomial dj\alpha(n+j)\,X^j.
\end{equation}
Thus,  the RH is equivalent to the hyperbolicity of the polynomials $J_{\gamma}^{d,n}(X)$
for all non-negative integers $d$ and $n$ \cite{CV,DL, Polya}.
Since this condition is preserved under differentiation, to prove RH
 it would be enough to show hyperbolicity for  the $J_{\gamma}^{d,0}(X)$
\footnote{The hyperbolicity for $J_{\gamma}^{d,0}(X)$ has been confirmed for $d\leq2\cdot10^{17}$
by Chasse (cf. Theorem 1.8 of \cite{Chasse}).}. 
 Due to the difficulty of proving RH, research has focused on establishing hyperbolicity 
for all shifts~$n$ for small~$d$. {Previous to this paper, hyperbolicity was known} for $d\leq 3$ by 
work\footnote{These works use a slightly different normalization for the $\gamma(n)$.} 
of Csordas, Norfolk, and Varga, and Dimitrov and Lucas \cite{CNV, DL}.

Asymptotics for the $\gamma(n)$ were obtained of Coffey and Pustyl'nikov \cite{Coffey, Pust}. We improve 
on their results by obtaining an arbitrary precision asymptotic formula\footnote{Our results 
imply the results in \cite{Coffey, Pust}  after typographical errors are corrected.} (see 
Theorem~\ref{AsympFn}), a result that is of independent interest.
 We will use this strengthened result to prove the following theorem for all degrees $d$.

\begin{theorem}\label{XiTheorem} 
If $d\geq 1$, then $J_{\gamma}^{d,n}(X)$ is hyperbolic for all sufficiently large $n$.
\end{theorem}

{An effective proof of Theorem~\ref{XiTheorem} for small $d$ gives the following theorem.

\begin{theorem}\label{effective}
If $1\leq d\leq \TBD$, then $J_{\gamma}^{d,n}(X)$ is hyperbolic for every $n\geq 0$.
\end{theorem}
}

Theorem~\ref{XiTheorem}  follows from a general phenomenon {that Jensen polynomials for a wide class of
sequences~$\alpha$} can be modeled by the {\it Hermite polynomials} $H_d(X)$, {which we define (in a somewhat 
non-standard normalization) as the orthogonal polynomials for the measure $\mu(X)=e^{-X^2/4}$ or more explicitly}
by the generating function
\begin{equation}\label{Def_Hermite}
{\sum_{d=0}^{\infty} H_d(X)\,\frac{t^d}{d!} \= e^{-t^2+Xt}} \= 
  1\ +X\,t  \+ (X^2-2)\,\frac{t^2}{2!}\+ (X^3-6X)\,\frac{t^3}{3!}\+\cdots
\end{equation}
More precisely, we will prove the following general theorem describing the limiting behavior of
Jensen polynomials of sequences with appropriate growth.

\begin{theorem}\label{GeneralTheorem2}
{Let $\{\alpha(n)\}$, $\{A(n)\}$ and $\{\delta(n)\}$ be three sequences of positive real numbers
 with $\delta(n)$ tending to zero and satisfying
\begin{equation}\label{log_alpha_ratio}
{ \log \Bigl({\frac{\alpha(n+j)}{\alpha(n)}\Bigr) \, \= \, A(n)j-\delta(n)^2 j^2
  + \text{\rm o}\bigl(\delta(n)^d\bigr)}} \qquad\text{\rm as $n\to\infty$}  \end{equation}
for some integer $d\ge1$ and all $0\le j\le d$. Then we have
\begin{equation}\label{tag2} \lim_{n\to\infty} \biggl(\frac{\delta(n)^{-d}}{\alpha(n)}
  \,J_\alpha^{d,n}\Bigl(\frac{\delta(n)\,X\m1}{\exp(A(n))}\Bigr)\biggr) \= H_d(X),\, 
  \end{equation}
uniformly for $X$ in any compact subset of~$\R$.}
  \end{theorem}
  
Since the Hermite polynomials have distinct roots, and since this property of a polynomial with real
coefficients is invariant under {small} deformation, {we immediately deduce the following corollary.}
\begin{corollary*}
\label{JHypCor}
The Jensen polynomials $J_{\alpha}^{d,n}(X)$ for a sequence $\alpha\colon\mathbb N\rightarrow\mathbb R$
satisfying the conditions in Theorem~$\ref{GeneralTheorem2}$ are hyperbolic {for all but finitely many 
values $n$}.  \end{corollary*}

\noindent
Theorem~\ref{XiTheorem} is a special case of this corollary. Namely, we shall employ 
Theorem~\ref{AsympFn} to prove that the Taylor coefficients $\{\gamma(n)\}$ satisfy the required 
growth conditions in Theorem~\ref{GeneralTheorem2} for every~$d\ge2$.

Theorem~\ref{GeneralTheorem2} in the case of the Riemann zeta function is the {\it derivative aspect 
Gaussian Unitary Ensemble} (GUE) random matrix model prediction for the zeros of Jensen polynomials. 
To make this precise, recall that Dyson, Montgomery, and Odlyzko \cite{KS,MO2,OD} conjecture that 
the non-trivial zeros of the Riemann zeta function are distributed like the eigenvalues of random 
Hermitian matrices.  These eigenvalues satisfy Wigner's Semicircular Law, as do the roots of the Hermite
polynomials $H_d(X)$, when suitably normalized,  as $d\rightarrow+\infty$ (see Chapter 3 of~\cite{RMB}). 
The roots of $J_{\gamma}^{d,0}(X)$, as $d\rightarrow+\infty$, approximate the zeros of 
$\Lambda\left(\frac{1}{2}+z\right)$ (see \cite{Polya} or Lemma 2.2 of \cite{CC}), and so GUE predicts 
that these roots also  obey the Semicircular Law.  Since the derivatives of $\Lambda\left(\frac{1}{2}+z\right)$ 
are also predicted to satisfy GUE, it is natural to consider the limiting behavior of $J_{\gamma}^{d,n}(X)$ 
as $n\rightarrow+\infty$.  The work here proves that these derivative aspect limits are the Hermite 
polynomials $H_d(X)$, which, as mentioned above, satisfy GUE in degree aspect.

Returning to the general case of sequences with suitable growth conditions,
Theorem~\ref{GeneralTheorem2} has applications in combinatorics where the
 hyperbolicity of polynomials determines the log-concavity of enumerative statistics.
For example, see the classic theorem by Heilmann and Leib \cite{HL}, along with works by Chudnovsky 
and Seymour, Haglund, Stanley, and Wagner \cite{CS, Haglund, HOW, Stanley, Wagner}, to name a few.
Theorem~\ref{GeneralTheorem2} represents a new criterion for establishing the hyperbolicity
of polynomials in enumerative combinatorics. The theorem reduces the problem to determining whether
suitable asymptotics hold.  Here we were motivated by a conjecture of Chen, Jia and Wang
concerning the Jensen polynomials $J_p^{d,n}${$(X)$},
where $p(n)$ is the partition function. 
Nicolas \cite{Nicolas} and Desalvo and Pak \cite{DP} proved that $J_p^{2,n}(X)$ is hyperbolic for $n\geq 25$, and more
recently, Chen, Jia, and Wang proved \cite{ChenJiaWang} that  $J_p^{3,n}(X)$ is hyperbolic for $n\geq 94$, 
inspiring them to state as a conjecture the following result.
\begin{theorem}[Chen-Jia-Wang Conjecture]
\label{Chen-Jia-Wang-Conj} For every integer $d\ge1$ there exists an integer~$N(d)$
such that $J_p^{d,n}(X)$ is hyperbolic for $n \geq N(d)$.
\end{theorem}
 
The table below gives the conjectured minimal value for $N(d)$ for $d\le32$. More
precisely, for each $d\le32$ it gives the smallest integer such that 
$J_p^{d,n}(X)$ is hyperbolic for $N(d)\le n\le 50000$.

\begin{table}[h]
\begin{tabular}{|r|cc|cc|cc|cc|cc|cc|cc|cc|}
\hline 
$d$    && 1 && 2 && 3  && 4 && 5 && 6 && 7 && 8  \\   \hline
$N(d)$ && 1 && 25 && 94 && 206 && 381 && 610 && 908 && 1269 \\
\hline
\hline 
$d$    && 9 && 10 && 11 && 12 && 13 && 14 && 15 && 16 \\   \hline
$N(d)$ && 1701 && 2210 && 2787 && 3455 && 4194 && 5018 && 5927 && 6917  \\
\hline
\hline 
$d$     && 17 && 18 && 19 && 20 && 21 && 22 && 23 && 24 \\   \hline
$N(d)$ && 8004 && 9171 && 10435 && 11788 && 13232 && 14777 && 16407 && 18146\\
\hline
\hline 
$d$     && 25 && 26 && 27 && 28 && 29 && 30 && 31 && 32  \\   \hline
$N(d)$  && 19975 && 21907 && 23938 && 26068 && 28305 && 30636 && 33084 && 35627  \\
\hline
\end{tabular} \end{table}

\begin{remark}
Larson and Wagner \cite{LarsonWagner} have made the proof of Theorem~\ref{Chen-Jia-Wang-Conj} effective by a brute 
force implementation of Hermite's criterion (see Theorem C of \cite{DL}). They showed that the values in the table 
are correct for $d=4$ and $d=5$ {and} that $N(d)\leq (3d)^{24d}(50d)^{3d^2}$ in~general. The true 
{values are presumably much smaller, and are probably of only polynomial growth, the numbers in the 
table being approximately of size $N(d)\approx10\,d^2\log d$.}
\end{remark}

{Theorem~\ref{Chen-Jia-Wang-Conj}} suggests a natural generalization.
As is well-known, the numbers $p(n)$ are the Fourier coefficients of a modular form, namely 
 \begin{equation}\label{PartitionGenFunction}
  \frac{1}{\eta(\tau)} \= \sum_{n=0}^{\infty}\,p(n)\,q^{n-\frac1{24}}
  \qquad(\,\Im(\tau)>0,\;q=e^{2\pi i\tau})\,, \end{equation}
where $\eta(\tau)=q^{1/24}\prod(1-q^n)$ is the Dedekind eta-function.  Theorem~\ref{Chen-Jia-Wang-Conj} 
is then an example of a more general theorem about the Jensen polynomials of the Fourier coefficients of 
an arbitrary {\it weakly holomorphic modular form}, which for the purposes of this article will
mean a modular form (possibly of fractional weight and with multiplier system) {with real Fourier coefficients} on the full modular
group $\SL_2(\Z)$  {that is holomorphic apart from} a pole of (possibly fractional) positive order at infinity.  If $f$ is such a form, we denote its Fourier expansion by\footnote{Note 
that with these notations we have $p(n)=a_f(n-\frac1{24})$ for $f=1/\eta$, but making this shift of 
argument is irrelevant for the applicability of Theorem~\ref{MFCase} to Theorem~\ref{Chen-Jia-Wang-Conj}, 
since the required asymptotic property 
 is obviously invariant under translations of~$n$.}

\begin{equation}\label{fFourier}
 f(\tau) \=  \sum_{n\in -m+\Z_{\ge0}}a_f(n)\,q^n \qquad(m\in\Q_{>0}\,,\;a_f(-m)\ne0) \end{equation}
Then we will prove the following theorem, which includes Theorem~\ref{Chen-Jia-Wang-Conj}.
\begin{theorem}\label{MFCase} If $f$ is a weakly holomorphic modular form as above, then for any fixed 
$d\geq1$ the Jensen polynomials $J_{a_f}^{d,n}(X)$ are hyperbolic for all sufficiently large~$n$.
\end{theorem}

Our results are proved by showing that each of the sequences of interest to us (the partition 
function, the Fourier coefficients of weakly holomorphic modular forms, and the Taylor coefficients
at~$s=\frac12$ of ${4}s(1-s)\La(s)$)    satisfies the hypotheses of Theorem~\ref{GeneralTheorem2}, which
we prove in Section~\ref{GeneralTheoremProof}.  Actually, in Section~\ref{GeneralTheoremProof} we prove a
more general result (Theorem~\ref{MoreGeneral}) that gives the limits of suitably normalized Jensen polynomials 
for an even bigger class of sequences having suitable asymptotic properties (but without necessarily the
corollary about hyperbolicity). Theorem~\ref{MFCase} giving the hyperbolicity for coefficients of modular
forms (and hence also for the partition function) is proved in Section~\ref{MFCaseProof}.  In 
Section~\ref{XiTheoremProof} we prove Theorem~\ref{AsympFn}, which gives an asymptotic formula to all 
orders for the Taylor coefficients of $\La(s)$ at~$s=\frac12$, and in Section~\ref{XiProofSection} we 
prove {Theorems~\ref{XiTheorem}  and \ref{effective}} for the Riemann zeta function by using these asymptotics to verify that
the hypotheses of Theorem~\ref{GeneralTheorem2} are fulfilled by the numbers~$\gamma(n)$.
We conclude in Section~\ref{examples} with some numerical examples.

\section*{Acknowledgements} \noindent
The authors thank the generosity of the Max Planck Institute for Mathematics in Bonn for its support and hospitality.
The authors thank William Y. C. Chen, Rick Kreminski, Hannah Larson, Steffen L\"obrich, Peter Sarnak, and Ian Wagner for discussions
related to this work. They also thank Jacques G\'elinas for bringing their attention to old work of Hadamard cited as a footnote in Section 4.

\section{Proof of Theorem \ref{GeneralTheorem2}}\label{GeneralTheoremProof}

We deduce Theorem~\ref{GeneralTheorem2} from {the following} more general result.

\begin{theorem}\label{MoreGeneral}
Suppose that  $\{E(n)\}$ and $\{\delta(n)\}$ are 
positive real sequences with $\delta(n)$ tending to~0, and that $F(t)=\sum_{i=0}^{\infty} c_it^i$
is a formal power series with complex coefficients.
For a fixed $d\geq 1,$ suppose that there are real sequences $\{C_0(n)\},\dots, \{C_{d}(n)\}$, with 
$\lim_{n\rightarrow +\infty} C_i(n)=c_i$ for $0\leq i\leq d$, such that for $0\leq j\leq d$  we have
\begin{equation}\label{tag4}
 {\frac{\alpha(n+j)}{\alpha(n)}\,E(n)^{-j}}
  \= \sum_{i=0}^d\, {C_i(n)\, \delta(n)^i j^i} \+ \text{\rm o}\bigl(\delta(n)^{d}\bigr) \qquad\text{as } n\to +\infty.
\end{equation}
Then the conclusion of Theorem~\ref{GeneralTheorem2} holds with $\,\exp(A(n))\,$ replaced by~$E(n)$ and $H_d(X)$ replaced by $H_{F,d}(X)$, 
where the polynomials $H_{F,m}({X})\in\Bbb C[x]$ are now defined either by the generating function 
${F(-t)}\,e^{Xt}=\sum H_{F,m}(X)\,t^m/m!$ or 
in closed form by $H_{F,m}{(X}):=m!\,\sum_{k=0}^m {(-1)^{m-k}}c_{m-k}\,X^k/k!\,$.
\end{theorem}

\begin{proof}[Proof of Theorem~\ref{MoreGeneral} and Theorem~\ref{GeneralTheorem2}] 
{After replacing $\exp(A(n))$ by $E(n)$,} the polynomial appearing on the left-hand side of~(\ref{tag2}) {becomes}
$$ \frac{\delta(n)^{-d}}{\alpha(n)}\,J_\alpha^{d,n}\Bigl(\frac{\delta(n)\,X\m1}{E(n)}\Bigr)
 \= \sum_{k=0}^d\binom dk\,\Biggl[\delta(n)^{k-d}\,
 \sum_{j=k}^d(-1)^{j-k}\,\binom{d-k}{j-k}\,\frac{\alpha(n+j)}{\alpha(n)E(n)^j}\Biggr]\;X^k\,.
$$ 
Since $0\leq j\leq d$, and since the error term in (\ref{tag4}) is ${\rm o}(\delta(n)^d),$
we may reorder summation {and find that the limiting value as $n\to +\infty$ of} the quantity in square brackets {satisfies}
$${\lim_{n\rightarrow +\infty} {\Biggl[} \sum_{i=0}^d C_i(n)}\,\delta(n)^{k-d+i}\,
\sum_{j=k}^d(-1)^{j-k}\,\binom{d-k}{j-k}\,j^i {\Biggr]} \=  {(-1)^{d-k}}(d-k)!\,c_{d-k}\,,
$$
{because} the inner sum, which is the $(d-k)$th difference of the polynomial
$j\mapsto j^i$ evaluated at $j=0$, vanishes for $i<d-k$ and equals $(d-k)!$ for~$i=d-k$. 
Theorem~\ref{MoreGeneral} follows{, and} Theorem~\ref{GeneralTheorem2} is just the special case 
$E(n)= e^{A(n)}$ {and} $F(t)=e^{-t^2}$. 
\end{proof}

\section{Proof of Theorem~\ref{MFCase}}\label{MFCaseProof}

Assume that $f$ is a modular form of (possibly fractional) weight~$k$ on $\text{SL}_2(\Z)$
(possibly with multiplier system) and with a pole of (possibly fractional) order~$m>0$ at
infinity, and write its Fourier expansion at infinity as in~\eqref{fFourier}.
It is standard, either by the circle method of Hardy--Ramanujan--Rademacher or by using Poincar\'e 
series (for example, see \cite{BOOK}), that the Fourier coefficients of~$f$ have the asymptotic form
\begin{equation}\label{asymptotics}
  a_f(n) = A_f\,n^{\frac{k-1}2}\,I_{k-1}(4\pi\sqrt{mn}) + \text O\bigl(n^C\,e^{2\pi\sqrt{mn}}\bigr) 
\end{equation}
as $n\to\infty$ for some non-zero constants $A_f$ (an explicit multiple of $a_f(-m)$) and~$C$,
where $I_\kappa(x)$ denotes the usual $I$-Bessel function.  In view of the expansion of Bessel functions at infinity, 
this implies that $a_f(n)$ has an asymptotic expansion to all orders in $1/n$ of the form 
$$a_f(n)\;\sim\;e^{4\pi\sqrt{mn}}\;n^{\frac{2k-3}4}\,\exp\Bigl(c_0\+\frac{c_1}n\+\frac{c_2}{n^2}\+\cdots\Bigr)$$
for some constants $c_0,\,c_1,\,\dots$ depending on~$f$ (and in fact only on~$m$ and $k$ if we normalize
the leading coefficient $a_f(-m)$ of~$f$ to be equal to~1).  This gives an asymptotic expansion 
\begin{equation}\label{mftag}
{\log \Bigl(} \frac{a_f(n+j)}{a_f(n)}{\Bigr)}\;\sim\; 4\pi\sqrt m\,\sum_{i=1}^\infty \binom{1/2}i\frac{j^i}{n^{i-\frac12}} 
  \+\frac{2k-3}4\sum_{i=1}^\infty \frac{(-1)^{i-1}j^i}{i\,n^i}
  \+ \sum_{i,k\ge1}c_k\binom{-k}i\frac{j^i}{n^{i+k}}
\end{equation}
valid to all orders in~$n$, and it follows  that the sequence $\{a_f(n)\}$ satisfies the hypotheses of
Theorem~\ref{GeneralTheorem2} with $A(n)=2\pi\sqrt{m/n}+\text O(1/n)$ and 
$\delta(n)=(\pi/2)^{1/2}m^{1/4}n^{-3/4}+\text O(n^{-5/4})$.
Theorem~\ref{MFCase} then follows from the corollary to Theorem~\ref{GeneralTheorem2}.

\section{Asymptotics for $\La^{(n)}\bigl(\frac12\bigr)$}\label{XiTheoremProof}

Previous work of Coffey \cite{Coffey} and Pustyl'nikov \cite{Pust}
offer asymptotics\footnote{It is interesting to note that Hadamard previously obtained rough estimates for these derivatives in 1893. His formulas are correctly reprinted on p. 125 of \cite{GuoWang}.} for the derivatives $\La^{(n)}\bigl(\tfrac12\bigr)$. Here we follow a slightly different approach and obtain effective asymptotics,
a result which is of independent interest. In order to describe our asymptotic expansion, we first give a formula for these derivatives in terms of an auxiliary function, whose asymptotic expansion we shall then determine.

Following Riemann, (cf. Chapter 8 of \cite{Davenport}) we have
  $$ \La(s)  = \int_0^\infty t^{\frac s2-1}\,\tz(t)\,dt
  = \frac1{s(s-1)} + \int_1^\infty \bigl(t^{\frac s2}+t^{\frac{1-s}2}\bigr)\,\tz(t)\,\frac{dt}t, $$
where $\tz(t)=\sum_{k=1}^\infty e^{-\pi k^2t}=\frac12(t^{-1/2}-1)+t^{-1/2}\tz(1/t)\,$. It follows that
  \begin{equation}\label{LamFn} \La^{(n)}\bigl(\tfrac12\bigr) = -\,2^{n+2}\,n! + \frac{F(n)}{2^{n-1}}
 \end{equation}
for $n>0$ (both are of course zero for $n$ odd), where $F(n)$ is defined for any real $n\ge0$ by
 \begin{equation}\label{FnDefn}  F(n) =  \int_1^\infty (\log t)^n\,t^{-3/4}\,\tz(t)\,dt \,.  \end{equation}
In particular if $n$ is a positive integer, then the Taylor coefficients $\gamma(n)$ defined in (\ref{TaylorXi}) satisfy
\begin{equation}\label{gamma_formula}
\gamma(n)=\frac{n!}{(2n)!}\cdot \left(8\binom {2n}2 \La^{(2n-2)} \bigl(\tfrac12\bigr)-\La^{(2n)} \bigl(\tfrac12\bigr)\right)=\frac{n!}{(2n)!} \cdot\,\frac{32 \binom {2n}2F(2n-2) - F(2n)}{2^{2n-1}}.
\end{equation}

\begin{theorem}
\label{AsympFn}
If $n>0$ then the function $F(n)$ defined by \eqref{FnDefn} is given to all orders in~$n$ by the asymptotic expansion
 $$  F(n) \;\sim\; \sqrt{2\pi}\;\frac{L^{n+1}}{\sqrt{(1+L)n-\frac3{4}L^2}}\;e^{L/4-n/L+3/4}
   \,\Bigl(1 + \frac{b_1}n + \frac{b_2}{n^2} + \cdots\Bigr) \qquad(n\to\infty)\,,$$
where $L=L(n)\approx\log\bigl(\frac{n}{\log n}\bigr)$ is the unique positive 
solution of the equation $n=L(\pi e^L+\frac34)$ and each coefficient~$b_k$ belongs to $\Bbb Q(L)$, 
the first value being $b_1=\frac{2L^4+9L^3+16L^2+6L+2}{24\,(L+1)^3}$.
\end{theorem}
\begin{Example*}
\label{AsympFnEx}
Here we illustrate Theorem~\ref{AsympFn}. The two-term approximation 
 $$  F(n) \;\approx\; \sqrt{2\pi}\;\frac{L^{n+1}}{\sqrt{(1+L)n-\frac3{4}L^2}}\;e^{L/4-n/L+3/4}
   \,\Bigl(1 + \frac{b_1}n \Bigr)
   =:\widehat F(n) $$
   is sufficiently strong for the proof of Theorem~\ref{XiTheorem}. In particular, Theorem~\ref{AsympFn} and \eqref{gamma_formula} imply
   \begin{equation}\label{gammaHatnDefn}
   \widehat\gamma(n):=\frac{n!}{(2n)!}2^{6-2n}\binom {2n}2\widehat F(2n-2) =\gamma(n)
   \left(1+O\left(\frac{1}{n^{2-\varepsilon}}\right)\right).
   \end{equation}
   Here are some approximations $\widehat\gamma(n)$ obtained from this expression by numerically computing $L$ using its defining equation above. 
This table illustrates the high precision of this formula.
   
\medskip

\begin{center}
\newcolumntype{d}[1]{D{.}{.}{#1}}
\begin{tabular}{|r|rd{0}l|ld{0}l|rll|}
\hline \rule[3.5mm]{0mm}{2mm} \rule[-2mm]{0mm}{2mm}
$n$
	&&& \hspace{-0mm} $\widehat{\gamma}(n)$
	&&& \hspace{-0mm}$\gamma(n)$
	&&& \hspace{-32mm}$\gamma(n)/\widehat\gamma(n)$ \\   \hline 
\rule[3mm]{0mm}{2mm} 
$10$ 
	&  \hspace{1mm}$\approx$ & 1.6313374394 & \hspace{18mm}$\times 10^{-17}$ 
	&  \hspace{1mm}$\approx$ & 1.6323380490 & \hspace{18mm}$\times 10^{-17}$ 
	&  \hspace{1mm}$\approx$ &  $1.000613367$& \\
$100$
	&  \hspace{1mm}$\approx$ &  6.5776471904 & \hspace{18mm}$\times 10^{-205}$ 
	&  \hspace{1mm}$\approx$ &  6.5777263785 & \hspace{18mm}$\times 10^{-205} $ 
	&  \hspace{1mm}$\approx$ &  $1.000012038$& \\
$1000$
	&  \hspace{1mm}$\approx$ &  3.8760333086 & \hspace{18mm}$\times 10^{-2567}$
	&  \hspace{1mm}$\approx$ &  3.8760340890 & \hspace{18mm}$\times 10^{-2567}$ 
	&  \hspace{1mm}$\approx$ &  $1.000000201$&\\
$10000$
	&  \hspace{1mm}$\approx$ &  3.5219798669 & \hspace{18mm}$\times 10^{-32265}$ 
	&  \hspace{1mm}$\approx$ &  3.5219798773 & \hspace{18mm}$ \times 10^{-32265}$ 
	&  \hspace{1mm}$\approx$ &  $1.000000002$&\\
 $100000$
 	&  \hspace{1mm}$\approx$ &  6.3953905598 & \hspace{18mm}$ \times 10^{-397097}$ 
 	&  \hspace{1mm}$\approx$ &  6.3953905601 & \hspace{18mm}$\times 10^{-397097}$
 	&  \hspace{1mm}$\approx$ &  $1.000000000$&\\
\hline
\end{tabular}

\end{center}
\end{Example*}
\begin{proof}[Proof of Theorem \ref{AsympFn}]
We approximate the integrand in \eqref{FnDefn} by $f(t)=(\log t)^nt^{-3/4}e^{-\pi t}$
(from now on we consider~$n$ as fixed and omit it from the notations). We have 
$t\frac d{dt}\log f(t)=\frac n{\log t}-\pi t-\frac34$, so $f(t)$ assumes its unique maximum at $t=a$, 
where $a=e^L$ is the solution in~$(1,\infty)$ of $$n=\Big(\pi a+\frac34\Big)\log a\,.$$ 
We can then apply the usual saddle point method. The Taylor expansion of $f(t)$ around~$t=a$ is given by
 $$  \frac{f((1+\la)a\bigr)}{f(a)} = \Bigl(1+\frac{\log(1+\la)}{\log a}\Bigr)^n(1+\la)^{-3/4}e^{-\pi\la a}
  = e^{-C\la^2/2}\Bigl(1+A_3\la^3+A_4\la^4+\cdots\Bigr)\,,  $$
where $C=(\e+\e^2)n-\frac34$  (here we have set $\e=\frac1{\log a}=L^{-1}$) and the $A_i$ ($i\ge3$) are polynomials 
of degree $\lfloor i/3\rfloor$ in~$n$ with coefficients in~$\Bbb Q[\e]$. This expansion is found by expanding  ${\log (f((1+\la)a\bigr))-\log (f(a))}$ in $\lambda$. The linear term vanishes by the choice of $a$, the quadratic term is $-C\lambda^2/2$, and the coefficients of the higher powers of $\lambda$ are all linear expressions in $n$ with coefficients in $\Bbb Q[\e]$. Exponentiating this expansion gives the claimed expression  for ${f((1+\la)a\bigr)}/{f(a)}$, where the dominant term of each $A_i$ is governed primarily  by the exponential of the cubic term of the logarithmic expansion. The first few $A_i$ are 
\begin{align*} & A_3=\Bigl(\frac\e3+\frac{\e^2}2+\frac{\e^3}3\Bigr)\,n \m\frac14\,, 
  \qquad A_4=-\Bigl(\frac\e4+\frac{11\e^2}{24}+\frac{\e^3}2+\frac{\e^4}4\Bigr)\,n + \frac3{16}\,, \\
 & A_5=\Bigl(\frac\e5+\frac{5\e^2}{12}+\frac{7\e^3}{12}+\frac{\e^4}2+\frac{\e^5}5\Bigr)\,n\m\frac3{20}\,, \\
 & A_6=\Bigl(\frac{\e^2}{18}+\frac{\e^3}6+\frac{17\e^4}{72}+\frac{\e^5}6+\frac{\e^6}{18}\Bigr)\,n^2
   \m\Bigl(\frac\e4+\frac{91\e^2}{180}+\frac{17\e^3}{24}+\frac{17\e^4}{24}+\frac{\e^5}2+\frac{\e^6}6\Bigr)\,n
   +\frac5{32}\,.
\end{align*}
Plugging in $t={(1+\la) a}$ immediately gives the asymptotic expansion
\begin{align*}
 \int_1^\infty f(t)\,dt 
   & = a\,f(a)\,\int_{-1+1/a}^\infty e^{-C\la^2/2}\Bigl(1+A_3\la^3+A_4\la^4+\cdots\Bigr)\,d\la \\
   & = a\,f(a)\,\sqrt{\frac{2\pi}C}\;
  \Bigl(1+\frac{3A_4}{C^2}+\frac{15A_6}{C^3}+\cdots+\frac{(2i-1)!!A_{2i}}{C^i}+\cdots\Bigr)\;.
  \end{align*}
(Here only the part of the integral with $C\la^2<B\log n$, where $B$ is any function of~$n$ going
to infinity as $n$ does, contributes.)
This equality and the expression in Theorem \ref{AsympFn} are interpreted as asymptotic expansions. Although these series themselves may not converge for a fixed $n$, 
we may truncate the resulting approximation at $O(n^{-A})$ for some $A>0$, and as  $n\rightarrow+\infty$ this approximation becomes true to the specified precision. 
 Substituting into this expansion the formulas for $C$ and~$A_i$ 
in terms of~$n$ we obtain the statement of the theorem with $F(n)$ replaced by the integral over~$f(t)$, 
with only $A_{2i}$ ($i\le3k$) contributing to~$b_k$.  But then the same asymptotic formula holds also
for $F(n)$, since the ratio $f(t)/\tz(t)=1+e^{-3\pi t}+\cdots$ is equal to $1+\text O(n^{-K})$ for any~$K>0$ 
for $t$ near~$a$.  
\end{proof}

\section{Proof of Theorems~\ref{XiTheorem} and \ref{effective}}\label{XiProofSection}

\subsection{Proof of Theorem~\ref{XiTheorem}} \  {For each $d\geq 1$,} we employ Theorem~\ref{GeneralTheorem2} with sequences $\{A(n)\}$ and $\{\delta(n)\}$ for which
\begin{equation}\label{NEED}
{\log\Bigl(\frac{\gamma(n+j)}{\gamma(n)}\Bigr) \, \= \, A(n)j
-j^2\delta(n)^2 +\sum_{i=3}^{d} g_i(n)j^i+ \text{\rm o}\bigl({\delta(n)^d}\bigr)}
\end{equation}
for all $0\leq j\leq {d},$ where $g_i(n)={\text {\rm o}}\bigl({\delta(n)^i}\bigr)$. 
Stirling's formula, (\ref{gamma_formula}),
and  \eqref{gammaHatnDefn} gives
\begin{equation}\label{GammaExpression}
\gamma(n)=
\frac{e^{n-2}n^{n+\frac12}(1+\frac{1}{12n}) L^{\tn}}{2^{\tn-3}{\tn}^{\tn+\frac12}(1+\frac{1}{12\tn})}
 \sqrt{\frac{2\pi}{K}}
\cdot \exp\left(\frac{L}{4} -\frac{\tn}{L}+\frac34 \right) \left(1+\frac{b_1(\tn)}{\tn}\right) \left(1+O\left(\frac{1}{n^{2-\varepsilon}}\right)\right),
\end{equation}
where $\tn:=2n-2$, $L:=L(\tn)$, and $K:=K(\tn):=\left(L(\tn)^{-1}+L(\tn)^{-2}\right)\tn-3/4$.
The $L(\tn)$ are values of a non-vanishing holomorphic function for $\Re (n)>1$, and so for $|j|<n-1$ we have the Taylor expansion
\begin{equation*}
\mathcal L(j;n):=\frac{L(\tn+2j)}{L(\tn)}=1+\sum_{m\geq 1}\ell_m(n)\frac{j^m}{m!}.
\end{equation*}
If $J=\lambda (n-1)$ with $-1<\lambda<1$, then the asymptotic  $L(n)\approx \log(\frac{n}{\log n})$ implies
$$
\lim_{n\to+\infty} \mathcal L(J,n) =\lim_{n\to+\infty} \frac{L\big(\tn(\lambda+1)\big)}{L(\tn)}=1.
$$
  In particular, 
we have $\ell_1(n)
=\frac{2}{K\cdot L^2}$ and $\ell_2(n)= \frac{-8(\tn-3/4L)(1+L/2)}{K^3\cdot L^5}$ and ${\ell_m(n)=o\big(\frac{1}{(n-1)^m}\big)}.$
By a similar argument applied to
 \begin{equation*}
\mathcal K(j;n):=\frac{K(\tn+2j)}{K(\tn)}=1+\sum_{m\geq 1}k_m(n) \frac{j^m}{m!} \ \ \ \
{\text {\rm and}}\ \ \ 
\mathcal B(j;n):=\frac{1+\frac{b_1(\tn+2j)}{\tn+2j}}{1+\frac{b_1(\tn)}{\tn}}=1+\sum_{m\geq 1}\beta_m(n)\frac{j^m}{m!},
\end{equation*}
 we find that $\beta_m(n)=o\big(\frac{1}{(n-1)^{m+1}}\big)$,
 $k_1(n)=\frac{2(L+1)}{K\cdot L^2}-\frac{2\tn (L+2)}{K^2L^4},$
 and $k_m(n)=o\big(\frac{1}{(n-1)^m}\big)$ for $m\geq 2$.

Let $R(j;n)$ be the approximation for $\gamma(n+j)/\gamma(n)$ obtained from (\ref{GammaExpression}).
We then expand $\log R(j;n)=:\sum_{m\geq 1}g_m(n)j^m$, with the idea that we will choose $A(n)\sim g_1(n)$ and $\delta(n)\sim \sqrt{-g_2(n)}$.
To this end, if $J=\lambda(n-1)$ for $-1<\lambda<1$, then a calculation reveals that
\begin{equation}\label{GmLimit}
-(1+\lambda)\log(1+\lambda)=\lim_{n\to +\infty}\frac{\log R(J;n){-J\log\left(\frac{nL^2}{4\tn^2}\right)-J}}{n-1}.
\end{equation}
Therefore, $g_m(n)=O\big((n-1)^{1-m}\big)$, and
algebraic manipulations give 
\begin{align*}
g_1(n)&=\log\left(\frac{nL^2}{4\tn^2}\right) +\tn \ell_1(n)\frac{L+1}{L}-\frac{2}{L} +\frac{\ell_1(n)\cdot L}{4}-\frac{k_1(n)}{2}+O\left(\frac{1}{n^{2-\varepsilon}}\right), \\
g_2(n)&=-\frac{1}{\tn}+\big(4\ell_1(n) +\tn \ell_2(n)\big)\frac{L+1}{2L}  -\tn \ell_1(n)^2\frac{L+2}{2L}+O\left(\frac{1}{n^{2-\varepsilon}}\right).
\end{align*}
Using the
 formulas for $\ell_1(n),$ $\ell_2(n),$ and $k_1(n)$ above, we  define
\begin{align}\label{g1n}
\delta(n):=\sqrt{\frac{1}{\tn}-\frac{2}{L^2\cdot K}} \ \ \ \ \ 
{\text {\rm and}}\ \ \ \ \ A(n):=\log \Big(\frac{nL^2}{4\tn^2}\Big) +\frac{L-1}{ L^2\cdot K}+\frac{\tn (L+2)}{L^4\cdot K^2}.
\end{align}
The bounds for the $g_m(n)$ and the asymptotics above imply {the $\rm{o}(1)$ error term in} (\ref{NEED}), and also that  for sufficiently large $n$ we have  $0<\delta(n)\to 0.$ Therefore, Theorem~\ref{GeneralTheorem2} applies, and its
corollary  gives Theorem~\ref{XiTheorem}.

\subsection{Sketch of the Proof of Theorem~\ref{effective}}

Let $A(n)$ and $\delta(n)$ be as in (\ref{g1n}). 
If we let
$$
\widehat{J}_{\gamma}^{d,n}(X):= \frac{\delta(n)^{-d}}{\gamma(n)}\cdot J_{\gamma}^{d,n}\Bigl(\frac{\delta(n)\,X\m1}{\exp(A(n))}\Bigr)=\sum_{k=0}^d \beta^{d,n}_k X^k,
$$
then Theorem~\ref{XiTheorem} implies that $\lim_{n\rightarrow +\infty}\widehat{J}_{\gamma}^{d,n}(X)= H_d(X)=:
 \sum_{k=0}^d h_k X^k$.
We have confirmed the hyperbolicity of  the $\widehat{J}_{\gamma}^{d,n}(X)$ for  $n\leq 10^6$ and $4\leq d\leq \TBD$ using
Hermite's  criterion (see Theorem C of \cite{DL}).  

Using this criterion, we also
chose
vectors
$\varepsilon_d:=(\varepsilon_d(d), \varepsilon_d(d-1),\dots, \varepsilon_d(0))$  of positive numbers
and signs $s_d, s_{d-1},\dots, s_0\in \{\pm 1\}$
for which $\widehat{J}_{\gamma}^{d,n}(X)$ is hyperbolic if
$0\leq s_k(\beta^{d,n}_k - h_k) < \varepsilon_d(k)$ for all $k$.
To make use of these inequalities, for positive integers $n$ and
$1\leq j\leq \TBD$, define real numbers $C(n,j)$ by
\begin{equation}\label{Cbounds}
\frac{\gamma(n+j)}{\gamma(n)e^{A(n)j}}\cdot e^{\delta(n)^2j^2}=1+\frac{C(n,j)}{n^{3/2}}.
\end{equation}
Using an effective form of (\ref{GammaExpression}),
it can be shown\footnote{It turns out that $\delta(6)$ is not real.}  that $0<C(n,j)<14.25$ for all $n\geq 7$ and $1\leq j \leq \TBD$.
Finally, we  determined numbers $M_{\varepsilon_d}$ 
for which the required inequalities hold for
$n\geq M_{\varepsilon_d}$. The proof follows from the fact that we found suitable choices for which $M_{\varepsilon_d}<10^6$.

\begin{example}
We illustrate the case of $d=4$ using
$\varepsilon_4:=(0.041, 1.384, 0.813, 7.313, 0.804).$
For $n\geq 100$ the odd degree coefficients satisfy $$0<\beta^{4,n}_3< 28~\delta(n)\ \ \
{\text {\rm  and}}\ \ \  -145.70
\delta(n)<\beta^{4,n}_1< 0,
$$
 while the even degree coefficients satisfy
\begin{align*}
1-16.05~\delta(n)^2<\beta^{4,n}_4< 1,& 
\ \ -12< \beta^{4,n}_2<-12+ 16.20~\delta(n),
\ \  12- 16.01~\delta(n)< \beta^{4,n}_0< 12.&
\end{align*}
It turns out that  $M_{\varepsilon_4}:=104<10^6$.
\end{example}


\section{Examples}\label{examples}
 For convenience, we let the
$\widehat{J}_{\alpha}^{d,n}(X)$ denote the polynomials which converge to $H_d(X)$ in (\ref{tag2}).
We now illustrate Theorem~\ref{MFCase} with
(\ref{PartitionGenFunction}),  where $m=1/24$ and $k=-1/2$.  Using (\ref{mftag}), we may choose
$A(n)=\frac{2\pi}{ \sqrt{24n-1}}-\frac{24}{24n-1}$ and 
$\delta(n)=\sqrt{\frac{12 \pi}{ (24n-1)^{3/2}} -\frac{288}{(24n-1)^2}   }.
$ Although the one-term approximations of (\ref{mftag}) given at the end of Section~\ref{MFCaseProof} also satisfy Theorem~\ref{GeneralTheorem2}, the two-term approximations converge more quickly and better illustrate the result.
 With this data we observe indeed that the
 degree $2$ and $3$ partition Jensen polynomials are modeled by
 $H_2(X)=X^2-2$ and $H_3(X)=X^3-6X$.

\smallskip
\begin{center}

\begin{tabular}{|l|ll|ll|}
\hline \rule[-3mm]{0mm}{8mm}
$n$       && $\widehat{J}_{p}^{2,n}(X)$           && $\widehat{J}_{p}^{3, n}(X)$ \\   \hline 
$ 100$ &&  $\approx 0.9993X^2+0.0731X-1.9568$ && $\approx 0.9981X^3+0.2072X^2-5.9270X+1.1420$\\
$ 200$ &&  $\approx 0.9997X^2+0.0459X-1.9902$ && $\approx 0.9993X^3+0.1284X^2-5.9262X-1.4818$\\
$ 300$ &&  $\approx 0.9998X^2+0.0346X-1.9935$ && $\approx 0.9996X^3+0.0965X^2-5.9497X-1.3790$\\
$ 400$ &&  $\approx 0.9999X^2+0.0282X-1.9951$ && $\approx 0.9998X^3+0.0786X^2-5.9621X-1.2747$\\
\hspace{3mm}\vdots && \hspace{25mm}\vdots &&\hspace{30mm}\vdots \\
$ 10^{8}$ &&  $\approx 0.9999X^2+0.0000X-1.9999$ && $\approx0.9999X^3+0.0000X^2-5.9999X-0.0529$\\
\hline
\end{tabular}
\end{center}

\medskip
\noindent
{The next table illustrates Theorem \ref{XiTheorem} for the Riemann zeta function
using (\ref{g1n}) in the case of degrees $2$ and $3$.}

\medskip
\begin{center}
\begin{tabular}{|l|ll|ll|}
\hline \rule[-3mm]{0mm}{8mm}
$n$
	&& $\ \ \ \ \ \ \ \ \ \ \ \ \ \  \widehat{J}_{\gamma}^{2, n}(X)$
	&& $\ \ \ \ \ \ \ \ \ \ \ \ \ \ \ \ \ \ \ \ \widehat{J}_{\gamma}^{3,n}(X)$  \\
\hline 
$100$
	&&  $\approx 0.9896X^2 + 0.3083X - 2.0199$
	&&  $\approx 0.9769X^3 + 0.7570X^2 - 5.8690X - 1.2661$ \\
$200$
	&&  $\approx 0.9943X^2 + 0.2271X - 2.0061$
	&&  $\approx 0.9872X^3 + 0.5625X^2 - 5.9153X - 0.9159$ \\
$300$
	&&  $\approx 0.9960X^2 + 0.1894X - 2.0029$
	&&  $\approx 0.9911X^3 + 0.4705X^2 - 5.9374X - 0.7580$ \\
$400$
	&&  $\approx 0.9969X^2 + 0.1663X - 2.0016$
	&&  $\approx 0.9931X^3 + 0.4136X^2 - 5.9501X - 0.6623$ \\      
\hspace{2.5mm}\vdots \  && \hspace{25mm}\vdots &&\hspace{40mm}\vdots \\
$10^8$
	&&  $\approx 0.9999X^2+0.0003X-2.0000$
	&&  $\approx 0.9999X^3+0.0009X^2-5.9999X - 0.0014$\\
\hline
\end{tabular}

\end{center}

\medskip
\noindent
Finally, we conclude with data for the degree 6 renormalized Jensen polynomials
$J_{\gamma}^{6,n}(X)$ which converge to $H_6(X)=X^6-30X^4+180X^2-120$.

\medskip
\begin{center}
\begin{tabular}{|l|ll|}
\hline \rule[-3mm]{0mm}{8mm}
$n$
	&&\ \ \ \ \ \ \ \ \ \ \ \ \ \ \ \ \ \ \ \ \ \ \ \ \ \ \ \ \ \ \ \ \ \ \ \  \ \ \ \ \ \ \ \ \ \ $\widehat{J}_{\gamma}^{6, n}(X)$
	 \\
\hline 
$100$
	&&  $\approx 0.912X^6+3.086X^5-24.114X^4-55.652X^3+133.109X^2+151.696X-85.419$\\
$200$
	&&  $\approx 0.950X^6+2.374X^5-26.625X^4-42.824X^3+153.246X^2+115.849X-100.510$\\

$300$
	&&  $\approx 0.965X^6+2.011X^5-27.608X^4-36.282X^3+161.084X^2+97.843X~-106.295$\\
$400$
	&&  $\approx 0.973X^6+1.780X^5-28.139X^4-32.111X^3+165.303X^2+86.428X~-109.388$\\     
\ \ \vdots && \hspace{50mm}\vdots \   \hspace{50mm}\vdots \\
$10^{10}$
	&&  $\approx 0.999X^6+0.000X^5-29.999X^4-0.008X^3~+179.999X^2~+0.020X~-119.999$\\
\hline
\end{tabular}

\end{center}


\begin{thebibliography}{BrStr}

\bibitem{Polya} G. P\'olya, \emph{\"Uber die algebraisch-funktionentheoretischen Untersuchungen von J. L. W. V. Jensen},
Kgl. Danske Vid. Sel. Math.-Fys. Medd. \textbf{7} (1927), 3-33.

\bibitem{CV} G. Csordas and R. Varga, \emph{Necessary and sufficient conditions and the Riemann hypothesis,} Adv. Appl. Math. {\bf 11} (1990), 328--357.

\bibitem{DL} D. K. Dimitrov and F. R. Lucas,
\emph{Higher order Tur\'an inequalities for the Riemann $\xi$-function}
Proc. Amer. Math. Soc. \textbf{139} (2010), 1013-1022.


\bibitem{Chasse} M. Chasse, \emph{Laguerre multiplier sequences and sector properties of entire functions,} Complex Var. Elliptic Equ. {\bf 58} (2011), 875--885. 



\bibitem{CNV} G. Csordas, T. S. Norfolk, and R. S. Varga, \emph{The Riemann hypothesis and the
Tur\'an inequalities}, Trans. Amer. Math. Soc. \textbf{296} (1986), 521-541.

\bibitem{Coffey} M. W. Coffey, \emph{Asymptotic estimation of $\xi^{(2n)}(1/2)$: On a conjecture of Farmer and Rhoades},
Math. Comp. \textbf{78} (2009), 1147-1154.


\bibitem{Pust} L. D. Pustyl'nikov, \emph{Asymptotics of the coefficients of the Taylor series of the $\xi(s)$ function}
(Russian),
Izv. Ross. Akad. Nauk Ser. Mat. \textbf{65} (2001), 93-106.

\bibitem{GuoWang} D. R. Guo and Z. X. Wang,
\emph{Special functions}, World Scientific Publishing, Singapore, 1989.

\bibitem{KS} N. Katz and P.  Sarnak, \emph{Zeros of zeta functions and symmetry}, Bull. Amer.  Math. Soc. \textbf{36} (1999), 1--26.

\bibitem{MO2} H. Montgomery, {\it The pair correlation of zeros of the zeta function}, Proc. Sym. Pure Math. \textbf{24} (1973), Amer. Math. Soc., Providence, 181--193.

\bibitem{OD} A. M. Odlyzko,  {\it The $10^{20}$-th zero of the Riemann zeta function and $70$ million of its neighbors,} AT\&T Bell Lab, publication (1989).

\bibitem{RMB} G. W. Anderson, A. Guionnet, and O. Zeitouni, {\it An introduction to random  matrices}, Cambridge Univ. Press, Cambridge (2010).


\bibitem{CC} T. Craven and G. Csordas, \emph{Jensen polynomials and the Tur\'an and Laguerre inequalities}, Pac.  J. Math. {\bf 136} no. 2 (1989), 241--260.

\bibitem{HL} O. Heilmann and E. Lieb, \emph{Theory of monomer dimer systems},
Comm. Math. Phys. \textbf{25} (1972), 190-232.


\bibitem{CS} M. Chudnovsky and P. Seymour, \emph{The roots of the independence
polynomial of a clawfree graph}, J. Comb. Th. Ser. B
\textbf{97} (2007), 350-357.


\bibitem{Haglund} J. Haglund, \emph{Further investigations involving rook polynomials
with only real zeros}, Euro. J. Comb. \textbf{91} (2000), 509-530.


\bibitem{HOW} J. Haglund, K. Ono, D. G. Wagner, \emph{Theorems and conjectures involving
rook polynomials with only real zeros}, Topics in Number Theory, (Ed. S. Ahlgren et. al.)
Kluwer Acad. Publ. 1999, 207-221.



\bibitem{Stanley} R. P. Stanley, \emph{Log-concave and unimodal sequences in algebra, combinatorics, and geometry}, Graph theory and its applications: East and West (Jinan, 1986),
Ann. New York Acad. Sci. \textbf{576} (1989), 500-535.



\bibitem{Wagner} D. G. Wagner, \emph{The partition polynomial of a finite set system},
J. Comb. Th. Ser. A \textbf{56} (1991), 138-159.


\bibitem{Nicolas} J.-L. Nicolas, \emph{Sur les entiers $N$ pour lesquels il y a beaucoup de groupes ab\'eliens d'order $N$}, Ann. Inst. Fourier \textbf{28} (1978), 1--16.


\bibitem{DP} S. Desalvo and I. Pak, \emph{Log-concavity of the partition function},
Ramanujan J. \textbf{38} (2015), 61-73.


\bibitem{ChenJiaWang} W. Y. C. Chen, D. X. Q. Jia, and L. X. W. Wang, \emph{Higher Order Tur\'an Inequalities for the Partition Function}, Trans. Amer. Math. Soc., accepted for publication.


\bibitem{LarsonWagner} H. Larson and I. Wagner, \emph{Hyperbolicity of the partition Jensen
polynomials}, Res. Numb. Th., accepted for publication.

\bibitem{BOOK} K. Bringmann, A. Folsom, K. Ono, and L. Rolen, {\it Harmonic Maass forms and mock modular forms: theory
and applications}, AMS Colloquium Series, 2017.

\bibitem{Davenport} {H. Davenport}, \emph{Multiplicative Number Theory}, Springer, New York, 2000.





\end{thebibliography}
\end{document}